\documentclass[12pt]{amsart}
\usepackage{xcolor}
\usepackage{wasysym}\usepackage{MnSymbol} %interior mult \invneg
\usepackage[normalem]{ulem}%\sout \uline
\usepackage{verbatim} %to comment out sections
\def\&{&=& }

\def\bar{\overline}
\def\spc{strictly pseudo-convex\ }

\def\nbd{neighborhood\ }
\def\iff{if and only if\ }
 
\def\tonezero{T^{1,0}}
\def\t01{T^{0,1}}
\def\lie{\pounds}

\def\ctm{{{\bf{C}}\otimes TM}} %used \bC instead
\def\ctmprime{{\bf{C}\otimes TM^\prime}}
\def\cth{{\bf{C}}\otimes H}

\def\bR{{\bf R}}

\def\bC{{\bf C}}

\def\be{\begin{equation}}
\def\ee{\end{equation}}
\def\bea{\begin{eqnarray}}
\def\eea{\end{eqnarray}}
\def\beastar{\begin{eqnarray*}}
\def\eeastar{\end{eqnarray*}}
\newtheorem{theorem}{Theorem}[section]
\newtheorem{conjecture}{Conjecture}[section]
\newtheorem{lemma}{Lemma}[section]
\newtheorem*{definition}{Definition}%%%%%%%%%%%%%%%%%%% use newtheorem* to eliminate numbering

\newtheorem{proposition}{Proposition}

\newtheorem*{example}{Example}

\makeatletter
\def\blfootnote{\xdef\@thefnmark{}\@footnotetext}%to create footnote without number
\makeatother

\newcommand{\abs}[1]{\lvert#1\rvert}
\numberwithin{equation}{section}  %  to have section appear in eq numberr

\begin{document}
\title[Conjecture]{A Conjecture of Trautman}\thanks{The present article is based on a talk by the author at the University  of Basilicata in Fall 2018.  The author would like to express his appreciation of the courtesy shown to him by Professors Barletta and Dragomir during his visit.}
\author{Howard Jacobowitz}
\address{Department of Mathematical Sciences, Rutgers University, Camden New Jersey\\jacobowi@rutgers.edu}%added at end of paper
\date{\today   }
%date appears at bottom of first page
%\markboth{\date{\today}}   %want date on each page  %did not work
\maketitle
%\blfootnote{The present article is based on a talk by the author at the University  of Basilicata in Fall 2018.  The author would like to express his appreciation of the courtesy shown to him by Professors Barletta and Dragomir during his visit.}
\section{}
\medskip
In 1998 the physicist Andre Trautman conjectured that a three-dimensional CR manifold is locally realizable \iff  its canonical bundle admits a closed nowhere zero section.  First  we review the relevant definitions and in the next section give the physical context.  In Section 3 we outline the earlier results in \cite{J} which had proved a weak version of the Conjecture.  

A {\bf CR structure} on a three-dimensional manifold $M$ is a two-plane distribution $H\subset TM$ and a fiber preserving  anti-involution $J:H\to H$.  We denote this structure by $(M,H,J)$.  It is often useful to extend $J$ by complex linearity to a map
\[
J: \cth \to \cth .
\]
Then $J$ is completely determined by the eigenspace corresponding to the eigenvalue $i$ (or to the eigenvalue -$i$).

An equivalent definition of a CR structure on a three-dimensional manifold may be given in terms of a complex line bundle:  A CR structures on $M$ is a line bundle $B\subset \ctm$ with the property that $B\cap \bar{B}$ contains only the zero section.  Then 
\[
H=\{ \Re Z\ :\ Z\in B\}
\]
is of rank 2 and $J$ is defined on ${\bC } \otimes H=B\oplus \overline{B}$ by setting
\[
J(Z)=iZ \mbox{  if }Z\in B
\]
and
\[
J(Z)=-iZ\mbox{  if  } Z\in \bar{B}.
\]
So for $X-iY\in B$
\[
JX=Y \mbox{  and  } JY=-X.
\]

{\bf{Example}  }
Let $ M^3\subset {\bC } ^2$ be a real hypersurface and let $J$ denote the usual operator on $\bR ^4$ giving the complex structure.  Set $H_p=T_pM\cap JT_pM$ for each $p\in M$.  Now $J$ acts on $H$ and (M,H,J) is a CR structure.  Or, to use the alternative definition, just   take
\[
B= T^{1,0}({\bC } ^2)\cap \ctm 
\]
where $T^{1,0}$ is the linear span of
\[
\bigg \{ \frac \partial {\partial z_1},\ \ \frac \partial {\partial z_2}\bigg\}
\]
(and   $T^{0,1}$ is the span of the conjugates).

So later we write $B=\tonezero (M)=\tonezero$ and write $T^{0,1}$ for $\bar{B}$.

The canonical bundle is another complex line bundle associated to a CR structure.  It is a subbundle of the second exterior product.  For a real hypersurface in ${\bC } ^2$ it is generated by the two-form $dz_1\wedge dz_2$ restricted to $M$.  More generally, if the CR structure is given by a complex line bundle $B$ then
\[
\Omega = \{ \omega \in {\bC }\otimes \Lambda ^2(TM)\ :\ i_b \omega =0 \mbox{ for all }b\in \bar{B}\} .
\]
The interior product $i_b\omega$ is given by $i_b\omega (X)=\omega (b\wedge X)$.
\begin{definition}
(M,H,J) is realizable in a \nbd of $p$ if there exist complex functions $f_1$ and $f_2$ such that 
\[
(X+iJX)f_k=0
\]
for all $X\in H$ and
\begin{eqnarray*}
F:M\to{\bC } ^2\ \ \ \ \ \\
x\to (\Re{f_1},\Im{f_1},\Re{f_2},\Im{f_2})
\end{eqnarray*}
is an embedding.
\end{definition}
It follows  upon identifying $M$ with its image $F(M)$ that the original structure $(M,H,J)$ coincides with the CR structure induced as in the Example.

We digress briefly to discuss higher-dimensional CR structures and return to this in Section \ref{higherdim}.

\begin{definition}
$(M^{2n+1},B)$ is a CR manifold if $B\subset \ctm$ is a vector subspace of rank $n$ with $B\cap \bar{B}=\{ 0\}$ and $[\Gamma  B ,\Gamma B ]\subset \Gamma B$.  I.e.,  the commutator of local sections of $B$ is always in $B$.

\end{definition}
More precisely, we have defined a CR manifold of hypersurface type.
\begin{definition}
$(M^{2n+1},B) $ is realizable if there is an embedding $F: M\to \bC ^{n+1}$ with, after identifying $M$ with $F(M)$,
\[
\tonezero (\bC ^{n+1} )\cap \bC \otimes TM = B.
\]

\end{definition}

The canonical bundle is now a complex line bundle in the exterior product $\Lambda ^{n+1}({\bC } \otimes TM^{2n+1}).$  Namely,

\begin{definition}%forgot end{definitiion} Ran anyhow but everything following was slanted
The canonical bundle is
\[
\Omega =\{ \omega \in {\bC}\otimes\Lambda ^{n+1}(TM)\ :\ i_v\omega =0,\  \forall v\in T^{0,1}\} .
\]
\end{definition}

\begin{definition} A function $f:M\to \bC $ is a CR function  if $Lf=0$ for all $L \in T^{0,1}$.

\end{definition}

\begin{lemma}
$M^{2n+1}$ is realizable in ${\bC } ^{n+1}$ if there exist CR functions $f_1,\ldots ,f_{n+1}
$ such that
\be \label{df}
df_1\wedge\ldots \wedge df_{n+1}\neq 0.
\ee
\end{lemma}
\begin{proof}
\begin{comment}
Note that \eqref{df} implies
\[
df_1\wedge\ldots \wedge f_n \wedge \bar{df_1}\wedge\ldots\wedge\bar{df_n}\neq 0
 \]
 \end{comment}
Let $L_1,\ldots L_n$ be a basis for $T^{0,1} $ and let $T$ be any nonzero vector transverse to $H$.  From \eqref{df} and using that the functions are CR, we have
\[
df_1\wedge\ldots \wedge df_{n+1}(\bar{L}_1,\ldots ,\bar{L}_n,T)\neq 0.
\]
So $df_jT\neq 0 $ for some $j$, say $j=n+1$, which now implies 

\[
df_1\wedge\ldots \wedge f_n\wedge df_{n+1} \wedge \bar{df_1}\wedge\ldots\wedge\bar{df_n}\neq 0 .
\] 
Thus
\[
F=(f_1,\ldots , f_{n+1})
\]
is a local embedding.  Indeed perhaps after multiplying $F$ by $i$, $F(M)$ has the form
\[
\Im z_{n+1}=f(z_1, \ldots ,z_n, \Re z_{n+1})
\]

\end{proof}

The realizability problem is quite subtle.  For instance, most three-dimensional $C^\infty$ CR structures are not locally realizable \cite{JT}, \cite{N}.  

Most realizability results in higher dimensions concern \spc CR structures.  
\begin{definition}
A CR structure $(M,B)$ is strictly pseudo-convex if the quadratic form
\[
L\in B\to [L,\bar{L}] \ \ \ mod \{ B\oplus \bar{B} \}
\]
is definite.

\end{definition}

Such structures are realizable if $dim\  M\geq 7$.  See \cite{Ak} and \cite {Ku}  for the original proofs and \cite {We} for a variation. 

Although, as we said, the general realizability problem is  subtle there are  two easy results.

\begin{proposition}  
Real analytic CR manifolds are  locally realizable.
\end{proposition}
A proof can be found, for instance, in \cite[page 22]{Ja}.
\begin{proposition}\label{transverse}
A CR manifold admiting a vector field $v$ transverse to $H$ and preserving the CR structure
is locally realizable.
\end{proposition}
To preserve the CR structure means that the Lie derivative in the direction of $v$ satisfies

\[
\lie _vT^{1,0}=T^{1,0}
\]

A generalization of this result is important in Section 3 and will be proved there.

\section
\medskip
We first wish  to explain the observation of \cite{RT} that a shear-free congruence of null geodesics on a four-dimensional manifold induces a three-dimensional   CR structure on a quotient manifold.

Let $M^4$ be a Lorentz manifold with metric $g$ and let $k$ be a null vector field, $g(k,k)=0$.  Let $K$ be the real line bundle generated by $k$.  Set
\[
K^\perp _p=\{ v\in T_pM:g(v,k)=0\} .
\]
Note that $K\subset K^\perp$ and that $K^\perp / K$ is an $R^2$ bundle on $M$.  Following the  notation  in \cite {T}, let  $n\in K^\perp $.  Denote the equivalency class of $n$ in $K^\perp /K$ by $[n]$ and use the same notation for $n\in {\bC } \otimes (K^\perp /K)=\bC\otimes K^\perp /\bC\otimes K$. 

\begin{lemma}
The metric $g$ induces a well-defined positive definite inner product on $K^\perp / K$.  
\end{lemma}
\begin{proof}
Let $[n_1]$ and $[n_2]$ belong to the fiber of $K^\perp / K$ over some point of $M$.   Define $g([n_1],[n_2])$ to be $g(n_1,n_2)$.  If $v_1$ and $v_2$ are different choices then $v_j=n_j+a_jk$ and so
\beastar
g(v_1,v_2)&=& g(n_1+a_1k,n_2 + a_2k)\\
&=& g(n_1,n_2)
\eeastar
since $k$ is a null vector and $n_j\in K^\perp$.  This shows that $g$ is well-defined.

To see that $g$ is definite, assume that for some $[n]$ we have 
\[
g([n],[n])\equiv g(n,n)=0.
\]
By the definitions of $k$ and $K^\perp$ we also have
\[
g(k,k)=0
\]
and 
\[
g(k,n)=0
.
\]
So either $n$ is a multiple of $k$ or $g$ vanishes on  a two-dimensional plane.  The second alternative is not possible for a Lorentz metric.  So $n=ak$ and thus $[n]=0$.  Hence $g$ is definite, and since it arises from a Lorentz metric it is positive definite.

\end{proof}

Fix an orientation for $K^\perp /K$ (this is not a problem, as long as we care only about local results)   and then let  $J: K^\perp/K\to K^\perp /K$ be the operation of rotation by $\pi /2$ radians with respect to the induced metric and orientation.   Finally, set 
\[
N=\{ n\in \bC\otimes K^\perp : J[n]= - i[n]\} .
\]
Note that $N$ is a two-dimensional complex vector bundle on $M$. Extend the inner product $g$ to $N$ as a complex linear form.  For $n_1=\xi +iJ\xi$ and $n_2=\eta +iJ\eta$ in $N$ we have
\beastar
g(n_1,n_2)&=& g(\xi ,\xi) ) +ig(J\xi , \eta) +ig(\xi ,J\eta ) -g(J\xi , J\eta)\\
&=& 0
\eeastar
since $J$ is rotation by $\pi /2$ radians.  So $N$ is said to be totally null.  On the other hand,
\[
g(n_1,\bar{n_1} ) = 2g(\xi ,\xi )\neq 0.
\]

We have
\[
N\subset {\bC }\otimes K^\perp \subset {\bC }\otimes TM
\]
and 
\[
N\cap \bar{N} =
{\bC }\otimes K, \ \ \ \\ \ \ N \plus \bar{N}=\bC \otimes K^\perp .
\]

Now consider the flow generated by the vector field $k$.  For small values of the time parameter, the orbit space is a three-dimensional manifold (again, for local results this is clear); call it $M^\prime$.  Without additional assumptions on $k$ the bundle $N$ does not project to a well-defined subbundle of  ${\bC } \otimes TM^\prime$.  Here is where physics enters.

We temporarily drop the assumption that $k$ is null.  
\begin{definition}  \cite[page 1426]{RT}   The vector field $k$ is said to be conformally geodesic if the associated flow preserves $K^\perp$ and $g(k,k)$ does not change sign.
\end{definition}
Note that this definition  depends only on the conformal class of $g$ and also that  in Riemannian geometry the condition on the flow  and $g(k,k)=c$ imply  $\nabla _kk=0$. 

  The flow condition may be rewritten as 
 \[
\lie _k K^\perp \subset K^\perp .
 \] 
and is equivalent to 
\be
g(k)\wedge\lie _kg(k)=0 \label{liegk}
\ee
where $g(k)$ is the one-form defined by $g(k) v= g(k,v)$.  
To see this equivalence, we first note that if $v$ is a vector field satisfying $g(k)v=0 $ then also $k(g(k)v)=0$ and so

\be\label{v}
(\lie _kg(k))v+g(k)\lie _k v=0.
\ee
We want to derive $g(k)\wedge \lie _kg(k)=0.$  
It is enough to show that 
\[
g(k)v=0\implies \lie _kg(k)v=0.
\]That is, if $g(k)$ and $\lie _kg(k)$ have the same kernel then these one-forms are linearly dependent.  So assume $\lie _k K^\perp\subset K^\perp$ and $g(k)v=0$.  
We now have
\[
g(k)v=0\Rightarrow v\in K^\perp \Rightarrow \lie _kv\in K^\perp \Rightarrow
g(k)\lie _k v=0
\Rightarrow \lie _kg(k)v=0
\]
where the last implication follows from \eqref{v}.

On the other hand, if $g(k)\wedge \lie _kg(k)=0$, then
\[
g(k)v=0 \Rightarrow (\lie _k g(k)v=0 \Rightarrow g(k)\lie _k v=0\Rightarrow \lie _kK^\perp\subset K^\perp .
\]

 We are interested in the case where $k$ is a null vector, $g(k,k)=0$.  When $k$ is null the foliation of $M$ by the integral curves of $k$ is  called a congruence of null geodesics.

The Lorentz metric $g$ induces a degenerate inner product on $K^\perp$ and therefore also a (degenerate) conformal structure.
\begin{definition}
A conformally geodesic vector field is shear-free if the associated flow preserves the conformal structure of $K^\perp$.
\end{definition}

The physical hypothesis that $k$ generates  a shear-free congruence of null geodesics also can be formulated in terms of the Lie derivative.

\begin{theorem}\cite {RToptical}
A vector field $k$ on a manifold $M^4$ with Lorentz metric $g$ generates a shear-free congruence of null geodesics if and only if
\begin{eqnarray}
g(k,k)=0  \label{0}\\
%g(k)\wedge \lie _kg(k)=0\\
\lie _k g=\lambda g +\phi \otimes g(k)\label{lieg}
\end{eqnarray}
where $\lambda$ is a function, $\phi$ is a one-form,  $g(k)$ is as defined above, and $\phi \otimes g(k)$ signifies the symmetric product constructed from the one-forms.
\end{theorem}
\begin{proof}
We first show that \eqref{lieg}, together with \eqref{0}, implies \eqref{liegk} and hence $k$ is a conformally geodesic vector field.  We start with the Leibniz rule:
\[
k(g(u,v))=\lie _kg(u,v)+g(\lie_k u,v)+g(u,\lie_kv).
\]
Setting $u=k$ and rearranging this becomes
\[
\lie _kg(k,v)=k(g(k,v))-g(k,\lie _kv).
\]
Further
\[
k(g(k,v))=k(g(k)v) =( \lie _k g(k))v +g(k,\lie_k v)
\]
and so
\[
(\lie _k g)(k,v)=(\lie _kg(k)) (v).
\]
Now
\beastar
(g(k)\wedge\lie _kg(k))(u,v)&=& (g(k)u)(\lie _k g(k) v)-(g(k)v)(\lie _kg(k)u)\\
&=& 
 (g(k)u) (\lie _k g)(k,v)-(g(k)v)(\lie _kg(k,u))\\
 &=&
 g(k,u) (\lambda g(k,v)+\phi\otimes g(k)(k,v))\\
& &\ \ \ \ \ \ \ \  -g(k,v)(\lambda g(k,u)+\phi\otimes g(k)(k,u))\\
 &=&
 0.
 \eeastar

We know that $g(k)\wedge \lie _kg(k)=0$ implies that $K^\perp $ is preserved and so $k$ is conformally geodesic.

To show that  the conformal class $g$ induces   on $K^\perp$ is constant along the flow on $M$ induced by $k$, we let $v\in K^\perp$ be constant along the flow.  So $g(k)v=0$ and $\lie _kv=0$.  Thus
\beastar
k(g(v,v)) &=& \lie g(v,v)\\
&=&(\lambda g+\phi\otimes g(k))(v,v)\\
&=&\lambda g(v,v).
\eeastar
This gives us an ordinary differential equation.  If local coordinates $(t,x)$ are introduced with $k=\partial _t$ then the equation has the form
\[
\frac {\partial f(t,x)}{\partial t}=\lambda (t,x)f(t,x)
\]
and the solutions are 
\[f(t,x)=\Lambda (t,x)f(0,x)
\]
for some function $\Lambda$.  Thus
\[
g(v,v)(t,x)=\Lambda (t,x)g(v,v)(0,x).
\]
This shows that the conformal class of the metric on $K^\perp$ does not change under the flow. 

Conversely,we want to show that if $k$ generates a shear-free congruence of null geodesics then there exist a scalar function $\lambda$ and a one-form $\phi$ satisfying \eqref{lieg}.  To see this, we start with a frame invariant along the orbits, labeled $e_1, e_2,e_3,e_4$ with $\{ e_1,e_2,e_3\}$ a basis for $K^\perp$ and $g(e_4,e_4)=0$.  Let $0\leq i\leq 3,\ 0\leq j\leq 3$. Note that $g(k)e_i =0$ and  $g(k)x_4\neq 0$.  For  $p\in M$  parametrize the orbit through $p$ by $t$.  Since the conformal class of $g$ on $K^\perp$ is constant
\[
g(e_i, e_j)|_t=\Lambda (t)g(e_i,e_j)|_p 
\]
Thus
\[  (\lie _kg)(e_i,e_j)|_p =(\lie _k \Lambda )g(e_i,e_j)|_p.
\]
Define
\beastar
\lambda & = & \lie _k \Lambda\\
\phi (e_i) &=& (g(k)e_4)^{-1}\Big((\lie _k g)(e_i,e_4))-\lie _k\lambda g(e_i,e_4)\Big),\ \ 0\leq i,j\leq 4.
\eeastar
We have for $0\leq i,j\leq 3$
\beastar
(\lambda g +\phi\otimes g(k))(e_i, e_j) &=&(\lie _k\Lambda )g(e_i, e_j) =( \lie _k g)(e_i,e_j),
\eeastar
while for $i\leq 4$ we have
\beastar
(\lambda g +\phi\otimes g(k))(e_i, e_4) &=& (\lie _k \Lambda)g(e_i,e_4) +\phi (e_i)g(k)e_4\\
&=&
(
\lie  _k \Lambda ) g(e_i,e_4)+\lie _k(g(e_i,e_4)- (\lie _k\Lambda)g(e_i,e_4)\\
&=&(\lie _k g)(e_i,e_4).
\eeastar
Thus
\[
\lambda g +\phi\otimes g(k) = \lie _k g .
\]

\end{proof}
Let $\pi$ denote the map of $M$ to the orbit space
\[
\pi : M\to M^\prime .
\]
\begin{lemma}
Under the conditions of the Theorem,  
$\pi _*(N)$ is  a complex line bundle $\bar{B}\subset {\bC } \otimes TM^\prime $
which satisfies  $B\cap \bar{B} =\{ 0\}.$  

\end{lemma}

\begin{proof}
 Since $K^\perp$ is itself invariant under the flow, $K^\perp / K$ projects to a well-defined two-plane distribution $H$ on $M^\prime$ and on $H$ we have a well-defined conformal class of metrics.  Thus $\bC \otimes TH$ splits into the eigenspaces of $J$
\[
\bC \otimes TH = B\oplus \bar{B}.
\]
with $\pi _* N=\bar{B}$.
\end{proof}

That is, the physical assumptions lead to a CR structure on the orbit space.  Further, as we now show, the same conditions provide a  two-form $F$ associated to $N$ which itself also passes down to $M^\prime$.  The interest in such a two-form comes from considerations of Maxwell's equations.  In classical physics, the components of the magnetic and electrical fields can be used to construct a real two-form $F$, called the Faraday tensor.  Then, in the absence of charge, Maxwell's equations become $dF=0$. Naturally, in relativistic physics the situation is more complicated.

To define $F$ we first find a basis for $N$.  
Let $\xi \in K^\perp$ and $\xi \notin K$.Choose any $\eta \in K^\perp $ such that $J[\xi ]=[\eta ]$.  Then $n= \xi +i\eta$ and $k$ form a basis for $N$.

Let $g(k)$, defined above, and $g(n)$, defined in the same way, be one-forms on $M$.  Set
\[
F=g(n)\wedge g(k).
\]

Note that $F$ is  nowhere zero  since the one-forms $g(n)$ and $g(k)$ are independent.  For example, $g(n)\bar{n} \neq 0$ while $g(k)\bar{n}=0$.

%for $\tau$ real and $\tau \notin K^\perp $, we have
%\beastar
%F(t\wedge \bar{n} &=&g(n,\tau )g(k,\bar{n})-g(n\bar{n} )g(k,\tau )\\
%&=& -g(n, \bar{n} )g(k,\tau )\\
%&\neq &0.
%\eeastar

The two-form $F$ is associated to $N$ in the following sense:
\begin{lemma}
$N=\{ v\in{ \ctm}\ :\ i_vF=0\} $.

\end{lemma}
\begin{proof}
We have $g(k,k)=0 $ because $k$ is null; $g(k,n)=0$ because $N\subset \bC \otimes K^\perp$; and $g(n,n)=0$ because $N$ is totally null.  So for our basis $i_kF=0$ and $i_nF=0$.  Thus
\[
N\subset \{ v\in{ \ctm}\ :\ i_vF=0\} .
\]
Now let $t\in \{ v\in{ \ctm}\ :\ i_vF=0\} $.  So
\[
g(n,t)g(k)-g(k,t)g(n)=0.
\]
The independence of $g(n)$ and $g(k)$ implies $t\in \bC\otimes K^\perp$ at some point of $M$.  Thus 
\[
t=\alpha n +\beta \bar{n}+\gamma k
\]
for constants $\alpha , \beta , \gamma $.  Since $g(n,t)=0$ and $g(n,\bar{n})\neq 0$, we see that $\beta =0$ and thus $t\in N$.
\end{proof}
We may use $F$ to define a two-form on $M^\prime$:  Let $t_1$ and $t_2$ be vectors in  $ \ctmprime $.  Lift $t_j $ to a vector $t_j+\alpha _j k$ in $\ctm$.  Then
\beastar
F((t_1+\alpha _1k)\wedge (t_2+\alpha _2 k) )&=& g(n,t_1+\alpha _1k)g(k, t_2 +\alpha _2 k)\\
&-& g(n,t_2+\alpha _2 k)g(k,t_1+\alpha _1k)\\
&=&g(n)\wedge g(k) (t_1\wedge t_2).
\eeastar
So $F$ evaluated on the lift is independent of choices and gives a well-defined two-form on $M^\prime$.  Call this form $F^\prime$.  For $t\in \bar{B}={\bC}\otimes T^{0,1}(M^\prime)$ the natural lift, also called $t$ is in $N$. Thus from the Lemma
\[
t\in \bC\otimes T^{0,1}(M^\prime) \implies  i_tF^\prime=0.
\]
Hence $F^\prime$ 
is    section of the canonical bundle of $M^\prime$ and is nowhere zero.

In summary, the local quotient of a  Lorentzian manifold under  a shear-free congruence of null geodesics is a CR manifold which has a nowhere zero  section of its canonical bundle.  This section being closed is related to Maxwell's equationa and so is a reasonable hypothesis for physicists.  We now repeat Trautman's conjecture.
\begin{conjecture}
If a CR manifold $M^3$ admits a nowhere zero closed section of its canonical bundle, then the CR structure is locally realizable.
\end{conjecture}
As we have seen, the converse is true even globally.

\section{\label{higherdim}}

A weak version of the conjecture is true and holds for all dimensions.

Functions satisfying
\[
df_1\wedge\ldots \wedge df_k\neq 0
\]
are called independent.  Functions satisfying 
\[
df_1\wedge\ldots \wedge df_k \wedge d\bar{f_1}\wedge\ldots \wedge d{\bar{f_k}} \neq 0
\]
are called strongly independent.
\begin{example}
The hyperquadric $Q^3\subset \bC ^2$ is defined by $\Im z_2=\abs{z_1}^2$.  The bundle $T^{0,1}$ is generated by 
\[
L=\partial _{\bar{z_1}}-iz\partial _u
\]
where $u=\Re z_2$.  The CR function $f=z$ is strongly independent; The function $f=u+i\abs{z}^2$ is independent, but not strongly independent (at the origin). 
\end{example}
The following theorem preceded the formulation of Trautman's Conjecture and establishes a weak form.
\begin{theorem}\cite{J}\label{Y}
If the CR structure $M^{2n+1}$ has $n$ strongly independent CR functions near $p$ and if the canonical bundle has a closed nowhere zero section then $M^{2n+1}$ is realizable in a neighborhood of $p$.
\end{theorem}
The proof depends on the following complex version of Proposition \ref{transverse}
\begin{proposition}
$M$ is realizable in a neighborhood of $p$ \iff there exists a complex vector field $Y$ near $p$ such that
\begin{itemize}
\item
$Y$ is transverse to $T^{1,0}\oplus T^{0,1}$
\item
$\lie _YT^{1,0}=T^{1,0}$.
\end{itemize}
\end{proposition}
Thus the existence of a real vector field such that $\lie _v T^{1,0}=T^{1,0}$ is very special (since most realizable CR structures do not have such a vector field) but the existence of such a complex vector field characterizes realizability.
\begin{proof}
We first prove the necessity.  So assume $M$ is realizable near $p$.  Without loss of generality we assume $p=0$ and $M$ is given as
\[
M=\{ (z_1,\ldots , z_{n+1}):\Im {z_{n+1} =\rho (z_1,\ldots ,z_n,\bar{z}_1,\ldots ,\bar{z}_{n-1},\Re{z_{n+1}}} \} .
\]
Define $\bar{Y}$ by 
\be \label{YY}
 dz_{n+1}(\bar{Y})=1,\ \ \ \ \ dz_j(\bar{Y})=d\bar{z}_j(\bar{Y})=0,\ \ 1\leq j\leq n.  
\ee
Note that $\bar{Y}$ (and also $Y$) is transverse to $T^{1,0}\oplus T^{0,1}$.  Set
\[
\omega = dz_1\wedge \ldots \wedge dz_{n+1}|_M.
\]
This is a nowhere zero closed section of the canonical bundle.  As a consequence of Cartan's formula
\[
\lie _v =di_v+i_vd
\]
we have
\[
\lie _{\bar{Y}}\omega =d(i_{\bar{Y}}\omega )+ i_{\bar{Y}}d\omega =0.
\]
This implies $\lie _{\bar{Y}} T^{0,1}=T^{0,1}$
 and so also
\[ \lie _Y T^{1,0}=T^{1,0}.
\]
Conversely, we will assume that $\lie _YT^{1,0}=T^{1,0}$ with $Y$ transverse to $T^{1,0}\oplus T^{0,1}$, and show that $M$ is locally realizable.  This is just a slight modification of a standard proof of Proposition \ref {transverse}.  Extend $Y$ and each of the vectors in $\tonezero$ to $\bC \otimes T(M\times \bR )$ by taking them constant in the $\bR $ direction. Let $Y$ still denote this extension and let $V$ denote the extension of the bundle $\tonezero$.  Set $Z$ to be the complex line bundle spanned by $Y+i\frac \partial {\partial t}$ where $t$ is the natural parameter for $\bR$.  Then  $W=V\oplus Z$ satisfies
\[
W\cap \bar{W}=\{ 0\} \mbox{    and    }W+\bar{W} = \bC\otimes T(M\times \bR) .
\]
Finally, as is easily seen, $W$ is closed under the commutation of vector fields,
\[
[\Gamma W,\Gamma W ] \subset \Gamma W.
\]
Thus $W$ satisfies the conditions of the Newlander-Nirenberg Theorem \cite {NN} and so defines a complex structure on $M\times \bR$.  Since $W\cap \bC\otimes TM=\tonezero(M\times R)$, the CR structure induced on $M$ is the one we started with.
\end{proof}
All that is left to do in the proof of Theorem \ref {Y} is to show that if $f_1,\ldots , f_n$ are CR functions on $M^{n+1}$ with 
\[
df_1\wedge\ldots \wedge d\bar{f}_n\neq 0
\]
and if $\omega$ is a nowhere zero section of the canonical bundle with
\[
d\omega =0
\]
then there is a complex vector field $Y$ with
\begin{itemize}
\item
$Y$ transverse to $T^{1,0}\oplus T^{0,1}$
\item
$\lie _YT^{1,0}=T^{1,0}$.
\end{itemize}
We just use the closed section to find a replacement for $dz_{n+1}$ in \eqref{YY}.  
Because we prefer to work with the canonical bundle and not its conjugate, we start, as in the Proposition,  by defining a vector field $\zeta$ and then let $Y=\bar{\zeta}$. 
Towards this end, let $\theta$ be a nowhere zero one-form annihilating $\tonezero\oplus \bar{\tonezero}$.  Then
\[
\theta \wedge df_1\wedge \ldots \wedge df_n
\]
is a nowhere zero section of the canonical bundle.  This bundle is one dimensional, so
\[
\omega = f \theta \wedge df_1\wedge \ldots \wedge df_n .
\]
Define $\zeta$ by 
\[
f\theta (\zeta)=1,\ \ \ \ df_j(\zeta)=0,\ \ \ d\bar{f}_j(\zeta)=0
\]
$\zeta$ can be thought of as a complex version of the Reeb vector field.  In particular, it is transverse to $T^{1,0}\oplus T^{0,1}$.

We have
\begin{eqnarray*}
\lie _\zeta \omega &=&d(i_\zeta\omega) + i_\zeta d\omega\\
&=&d(f\theta (\zeta) )df_1\wedge\ldots \wedge df_n)+i_\zeta d\omega\\
&=&0.
\end{eqnarray*}
\begin{lemma}
If 
$ \lie _\zeta \omega =0 $ then $ \lie _\zeta T^{0,1}=T^{0,1}$.
\end{lemma}
\begin{proof}
We have for all vector fields $\zeta$ and $v$ and all forms $\omega$
\[
\lie _\zeta i_v \omega  = i_{\lie _\zeta v}\  \omega+i_v \lie _\zeta \omega .
\]
So, if $v\in T^{0,1}$, hence $i_v\omega=0$, and $\lie _\zeta \omega =0$, then
\[
i_ {\lie _\zeta v}\ \omega =0
\]
and so $\lie _\zeta v$ is also in $T^{0,1}$.
\end{proof}
This Lemma has a partial converse:  If $\lie _\zeta T^{0,1}=T^{0,1}$ then $\lie \omega =\alpha \omega$ for some function $\alpha$.

Finally, we set $Y={\bar{\zeta}}$.  Thus, $Y$ is transverse to $T^{1,0}\oplus T^{0,1}$ and
\[
\lie _Y \tonezero = \bar{\lie _\zeta T^{0,1}} = \tonezero 
\]
and we are done.

\bibliographystyle{amsplain}

\end{document}